\documentclass[reqno,12pt]{amsart}
\usepackage{amscd,amssymb,comment,euscript}
\usepackage{enumerate}

\usepackage[initials]{amsrefs}

\theoremstyle{plain}
\newtheorem{thm}{Theorem}[section]
\newtheorem{lem}[thm]{Lemma}

\newtheorem{prop}[thm]{Proposition}

\theoremstyle{remark}

\theoremstyle{definition}

\usepackage{color}

\newcommand\Ac{{\mathcal{A}}}
\newcommand\Cpx{{\mathbb C}}
\newcommand\Dc{{\mathcal{D}}}
\newcommand\dom{\operatorname{dom}}
\newcommand\Ec{{\mathcal{E}}}
\newcommand\eps{\epsilon}
\newcommand\HEu{{\EuScript H}}                   
\newcommand\Ints{{\mathbb Z}}
\newcommand\Lc{{\mathcal{L}}}
\newcommand\Mcal{{\mathcal{M}}}
\newcommand\meas{{\omega}}
\newcommand\measspace{{Z}}
\newcommand\Nats{{\mathbb N}}
\newcommand\Reals{{\mathbb R}}
\newcommand\Tcirc{{\mathbb T}}
\newcommand\zbar{{\overline z}}

\begin{document}

\title[Reduction theory and Brown measure]{On reduction theory and Brown measure for closed unbounded operators}

\author[Dykema]{K. Dykema$^*$}
\address{Ken Dykema, Department of Mathematics, Texas A\&M University, College Station, TX, USA.}
\email{ken.dykema@math.tamu.edu}
\thanks{\footnotesize ${}^{*}$ Research supported in part by NSF grant DMS--1202660.}
\author[Noles]{J. Noles$^*$}
\address{Joseph Noles, Department of Mathematics, Texas A\&M University, College Station, TX, USA.}
\email{jnoles@math.tamu.edu}
\author[Sukochev]{F. Sukochev$^{\S}$}
\address{Fedor Sukochev, School of Mathematics and Statistics, University of new South Wales, Kensington, NSW, Australia.}
\email{f.sukochev@math.unsw.edu.au}
\thanks{\footnotesize ${}^{\S}$ Research supported by ARC}
\author[Zanin]{D. Zanin$^{\S}$}
\address{Dmitriy Zanin, School of Mathematics and Statistics, University of new South Wales, Kensington, NSW, Australia.}
\email{d.zanin@math.unsw.edu.au}

\subjclass[2000]{47C99 (47C15)}

\keywords{Direct integral, Brown measure}

\begin{abstract}
The theory of direct integral decompositions of both bounded and unbounded operators is further developed;
in particular, results about spectral projections, functional calculus and affiliation to von Neumann algebras are proved.
For operators belonging to or affiliated to a tracial von Neumann algebra that is a direct integral von Neumann algebra, the Brown
measure is shown to be given by the corresponding integral of Brown measures.
\end{abstract}

\date{September 10, 2015}

\maketitle

\section{Introduction}

Reduction theory is a way of decomposing von Neumann algebras
as direct integrals (a generalization of direct sums) of other von Neumann algebras.
It is commonly employed,
when the direct integral decomposition is done over the center of the von Neumann algebra,
to see that an arbitrary von Neumann algebra is a direct integral of factors.
However, the direct integral decomposition can be done over any von Neumann subalgebra of the center.

Our main goal in this paper is to show that, in the context of tracial von Neumann algebras
and certain unbounded operators affiliated to such von Neumann algebras,
the Brown spectral distribution measure behaves well with respect to direct integral decompositions.
This result (Theorem~\ref{thm:decompBrown}) is a natural development and its proof is technically nontrivial.
This result finds immediate application in the paper~\cite{DSZ:unbdd},
that extends results from~\cite{DSZ} about Schur upper-triangular forms to certain unbounded operators affiliated to finite von Neumann algebras.

We will now describe some of the theory of Brown measure and the Fulgelde--Kadison determinant, on which it depends.
Given a tracial von Neumann algebra $(\Mcal,\tau)$,
by which we mean a von Neumann algebra $\Mcal$ and a normal, faithful, tracial state $\tau$,
the {\em Fuglede--Kadison determinant}~\cite{FK52} is the map $\Delta=\Delta_\tau:\Mcal\to[0,\infty)$ defined by
\[
\Delta(T)=\exp\big(\tau(\log|T|)\big):=\lim_{\eps\to0^+}\exp\big(\tau(\log|T|+\eps)\big).
\]
Fulglede and Kadison proved that it is multiplicative:  $\Delta(AB)=\Delta(A)\Delta(B)$.

The {\em Brown measure} $\nu_T$ was introduced by L.G.\ Brown~\cite{B86}.
It is a sort of spectral distribution measure for elements $T\in\Mcal$ (and for certain unbounded operators affiliated to $\Mcal$).
It is defined to be the Laplacian (in the sense of distributions in $\Cpx$) of the function $f(\lambda)=\frac1{2\pi}\log\Delta(T-\lambda)$;
Brown proved, among other properties, that it is a probability measure whose support is contained in the spectrum of $T$.
Later, Haagerup and Schultz~\cite{HS07} proved that
the Fuglede-Kadison determinant and Brown measure are defined and have nice properties for all
closed, densely defined, possibly unbounded operators $T$ affiliated to $\Mcal$ such that
$\tau(\log^+|T|)<\infty$, where $\log^+(x)=\max(\log(x),0)$.
We will use the notation $\exp(\Lc_1)(\Mcal,\tau)$ for this set.
It is easy to see that $\exp(\Lc^1)(\Mcal,\tau)$ is an $\Mcal$-bimodule;
it is, in fact, a $*$-algebra containing $\Mcal$ as a $*$-subalgebra (see~\cite{DSZ:unbdd}).
A characterization (Theorem~2.7 of~\cite{HS07})
of the Brown measure $\nu_T$ of $T\in\exp(\Lc^1)(\Mcal,\tau)$ is as the unique probability measure on $\Cpx$
satisfying
\begin{equation}\label{eq:char1}
\int_{\Cpx}\log^+|z|\,d\nu_T(z)<\infty
\end{equation}
and
\begin{equation}\label{eq:char2}
\int_{\Cpx}\log|z-\lambda|\,d\nu_T(z)=\log\Delta(T-\lambda)\qquad(\lambda\in\Cpx).
\end{equation}

Brown measure is naturally defined on elements of $\exp(\Lc^1)$;
we will need reduction theory also for unbounded operators in Hilbert space.
Nussbaum~\cite{Nu64} introduced this theory and developed several aspects of it.
In this paper, we will prove and make use of some further results about direct integral decompositions of unbounded operators, for example,
about (a) functional calculus for decomposable unbounded self-adjoint operators, (b) polar decompositions and (c) affiliated operators.

The outline of the rest of the paper is as follows:
Section~\ref{sec:prelims} contains preliminaries about reduction theory, gleaned from the text~\cite{Dix81} of Dixmier
and the paper~\cite{Nu64} of Nussbaum;
Section~\ref{sec:bdd} contains results about reduction theory for bounded operators, including results about spectral projections and functional calculus;
Section~\ref{sec:unbdd} contains results about reduction theory for unbounded operators, including functional calculus, polar decomposition
and affiliattion of operators to von Neumann algebras;
Section~\ref{sec:Brown} contains our main result about Brown measure of a decomposable operator.

\section{Preliminaries about reduction theory}
\label{sec:prelims}

In this section, we will recall elements of the reduction theory for von Neumann algebras as expounded by Dixmier~\cite{Dix81}
and some definitions and results from Nussbaum's paper~\cite{Nu64} on reduction theory for unbounded operators.
Throughout, we let $\meas$ be a fixed $\sigma$-finite positive measure on a standard Borel space $\measspace$,
namely a Polish space endowed with the Borel $\sigma$-algebra.

\subsection{Direct integrals of Hilbert spaces}
\label{subsec:HilSp}
A {\em measurable field} of Hilbert spaces is a function $\zeta\mapsto\HEu(\zeta)$, ($\zeta\in \measspace$), where each $\HEu(\zeta)$ is a Hilbert space,
together with a set $S$ of vector fields (namely, functions $\zeta\mapsto x(\zeta)\in\HEu(\zeta)$) that are said to be measurable
and that satisfy
\begin{enumerate}[(i)]
\item that the function $\zeta\mapsto\langle x(\zeta),y(\zeta)\rangle$ is measurable for all $x,y\in S$ and
\item if $v$ is a vector field and the function $\zeta\mapsto\langle x(\zeta),v(\zeta)\rangle$ is measurable for each $x\in S$, then $v\in S$.
\end{enumerate}
The {\em direct integral Hilbert space}
\[
\HEu=\int_\measspace^\oplus\HEu(\zeta)\,d\meas(\zeta)
\]
consists of all measurable vector fields $x\in S$
for which the function $\zeta\mapsto\|x(\zeta)\|^2$ is integrable with respect to $\meas$.
The inner product on $\HEu$ is given by
\[
\langle x,y\rangle=\int_\measspace\langle x(\zeta),y(\zeta)\rangle\,d\meas(\zeta).
\]
See~\cite{Dix81} Sections II.1.1-II.1.5.

\subsection{Fields of Bounded Operators}
\label{subsec:BddOps}
A field $\zeta\mapsto T(\zeta)\in B(\HEu(\zeta))$ ($\zeta\in\measspace$) of bounded operators is said to be {\em measurable} if for every measurable vector field $x\in S$
(as in~\ref{subsec:HilSp}) the field $\zeta\mapsto T(\zeta)x(\zeta)$ is measurable.
In this case, the map $\zeta\mapsto\|T(\zeta)\|$ is measurable.
See~\cite{Dix81} Section II.2.1.

\subsection{Decomposable and Diagonal Bounded Operators}
\label{subsec:ddBdd}
If $T$ is a measurable field of bounded operators as in~\ref{subsec:BddOps} and if the map
\begin{equation}\label{eq:Tzeta}
\zeta\mapsto\|T(\zeta)\|
\end{equation}
is essentially bounded, where $\|\cdot\|$ is the operator norm,
then $T$ describes a bounded linear operator, also denoted by $T$, on the direct integral Hilbert space $\HEu$, by $(Tx)(\zeta)=T(\zeta)x(\zeta)$,
and we write
\begin{equation}\label{eq:Tdecomp0}
T=\int_\measspace^\oplus T(\zeta)\,d\meas(\zeta).
\end{equation}
The norm of $T$ equals the essential supremum of the map~\eqref{eq:Tzeta}.
Such operators $T$ on $\HEu$ are said to be {\em decomposable}.
The set of decomposable operators, which we will denote $\Ec$, is a subalgebra of $B(\HEu)$ and the $*$-algebra
operations have the obvious almost-everywhere-pointwise
interpretation --- see~\cite{Dix81} Section II.2.3.
In particular, $T$ is self-adjoint if and only if $T(\zeta)$ is self-adjoint for almost every $\zeta$
and $T\ge0$ if and only if $T(\zeta)\ge0$ for almost every $\zeta$.
The {\em diagonal operators} are the decomposable operators $T$ for which each $T(\zeta)$ is a scalar multiple of the identity operator on $\HEu(\zeta)$.
The algebra of all diagonal operators, which we shall denote $\Dc$, is a von Neumann algebra isomorphic to $L^\infty(\measspace,\meas)$, and its commutant is the von Neumann
algebra $\Ec$ of decomposable operators  --- see~\cite{Dix81} Sections II.2.4 and II.2.5.

\subsection{Fields of von Neumann algebras}
\label{subsec:vN}
All of the von Neumann algebras considered in this paper will be assumed to be countably generated.
If $\Ac$ is a von Neumann algebra in $B(\HEu)$ that is generated by the algebra $\Dc$ of diagonalizable operators together with a countable set $\{T_i\mid i\ge1\}$
of decomposable operators, then $\Ac$ is said to be {\em decomposable}.
Letting $\Ac(\zeta)$ be the von Neumann algebra in $B(\HEu(\zeta))$ generated
by $\{T_i(\zeta)\mid i\ge1\}$, we have that whenever $T$ is a decomposable operator, then $T\in\Ac$ if and only if $T(\zeta)\in\Ac(\zeta)$ for almost every $\zeta$.
We write
\[
\Ac=\int_\measspace^\oplus\Ac(\zeta)\,d\meas(\zeta).
\]
Note that the von Neumann algebra $\Dc$ of diagonal operators is contained in the center of $\Ac$.
See~\cite{Dix81} Sections II.3.1 to II.3.2.

\subsection{Measurable fields of traces}
Suppose $\Ac=\int_\measspace^\oplus\Ac(\zeta)\,d\meas(\zeta)$ is a decomposable von Neumann algebra and
$\zeta\mapsto\tau_\zeta$ is a field of traces, each $\tau_\zeta$ being a trace on $\Ac(\zeta)^+$ taking values in $[0,+\infty]$.
The field of traces is said to be {\em measurable} if for every $T=\int_\measspace^\oplus T(\zeta)\,d\meas(\zeta)\in\Ac$,
the function $\zeta\mapsto\tau_\zeta(T(\zeta))$ is measurable.
In this case
\[
\tau=\int_\measspace^\oplus\tau_\zeta\,d\meas(\zeta)
\]
denotes the trace on $\Ac^+$ defined as follows.
When $T\in\Ac^+$, writing $T$ as in~\eqref{eq:Tdecomp0}, we have
\[
\tau(T)=\int_\measspace\tau_\zeta(T(\zeta))\,d\meas(\zeta).
\]
See~\cite{Dix81} Section II.5.1.

\subsection{Direct integral decomposition of  a finite von Neumann algebra and trace}
\label{subsec:finvN}
If $\Ac=\int_\measspace^\oplus\Ac(\zeta)\,d\meas(\zeta)$ is a decomposable von Neumann algebra and $\tau$ is a normal, faithful, tracial state on $\Ac$,
then there is a measurable field $\zeta\mapsto\tau_\zeta$ of normal, faithful, finite traces $\tau_\zeta$ on $\Ac(\zeta)$, so that 
\[
\tau=\int_\measspace^\oplus\tau_\zeta\,d\meas(\zeta).
\]
After redefining $\meas$, if necessary, we may without loss of generality assume each $\tau_\zeta$ is a tracial state.
See the Corollary in~\cite{Dix81} Section II.5.2.

\subsection{Measurable fields of unbounded operators}
\label{subsec:Nmeas}
We will denote the domain of a closed (possibly unbounded) operator $T$ on a Hilbert space by $\dom(T)$.
Let $\zeta\mapsto T(\zeta)$ be a field of closed operators on $\HEu(\zeta)$.
Let $P(\zeta)=(P_{ij}(\zeta))_{1\le i,j\le 2}\in M_2(B(\HEu(\zeta))$ be the projection onto the graph of $T(\zeta)$.
Nussbaum~\cite{Nu64} introduced the following notion of measurability:
the field of operators is {\em measurable} if for all $i$ and $j$, the field $P_{ij}(\zeta)$ of bounded operators
is measurable, in the sense of~\ref{subsec:BddOps}.
Proposition~6 of~\cite{Nu64} shows that in the case of an essentially bounded field of bounded operators,
measurablility in the above sense is equivalent to measurability as found in~\ref{subsec:BddOps}.
The field $\zeta\mapsto T(\zeta)$ is said to be {\em weakly measurable} if for every measurable vector field $\zeta\mapsto x(\zeta)$
of vectors such that for all $\zeta$, $x(\zeta)\in\dom(T(\zeta))$, the vector field $\zeta\mapsto T(\zeta)x(\zeta)$ is measurable.
Nussbaum proves (Corollary 2 of~\cite{Nu64}) that every measurable field $\zeta\mapsto T(\zeta)$ of closed operators is weakly measurable,
while the converse statement was shown to be false in~\cite{GGST12}.

\subsection{Decomposable unbounded operators}
\label{subsec:decompunbdd}
Given a measurable field $\zeta\to T(\zeta)$ of closed operators as in~\ref{subsec:Nmeas}
and letting $\HEu=\int_\measspace^\oplus\HEu(\zeta)\,d\meas(\zeta)$ be the direct integral Hilbert space,
define the operator $T$ to have domain equal to the set of vectors $x\in\HEu$ defined
by square integrable vector fields $\zeta\mapsto x(\zeta)$ such that $x(\zeta)\in\dom(T(\zeta))$ for all $\zeta$ and such that the vector field
$\zeta\mapsto T(\zeta)x(\zeta)$ is square integrable, and for such an $x$ to have value $Tx$ equal to the vector field
\[
(Tx)(\zeta)=T(\zeta)x(\zeta).
\]
By Proposition~7 of~\cite{Nu64}, $T$ is a closed operator.
A closed operator that arises in this way from a measurable field of closed operators is said to be {\em decomposable}.
By Corollary~4 of~\cite{Nu64}, a closed operator in $\HEu$ is decomposable if and only if it
permutes\footnote{
Using notation that seems to be out of fashion but was once standard, we say that a bounded operator $S$ {\em permutes} with a closed, possibly unbounded operator $T$
if $S(\dom(T))\subseteq\dom(T)$ and $TSx=STx$ for all $x\in\dom(T)$.}
with all the bounded diagonalizable operators, as described in~\ref{subsec:ddBdd}.
By Theorem~2 of~\cite{Nu64}, a closed operator in $\HEu$ is decomposable if and only if 
it is affiliated with the von Neumann algebra $\Ec$ of all bounded decomposable operators.

Suppose $T=\int_\measspace^\oplus T(\zeta)\,d\meas(\zeta)$ is a decomposable closed operator.
By Theorem~3 of~\cite{Nu64}:
\begin{enumerate}[(a)]
\item $T(\zeta)$ is densely defined in $\HEu(\zeta)$ for almost every $\zeta$ if and only if $T$ is densely defined in $\HEu$;
\item $T(\zeta)$ is self-adjoint for almost every $\zeta$ if and only if $T$ is self-adjoint;
\end{enumerate}

\section{Spectral projections and functional calculus for bounded operators}
\label{sec:bdd}

In this section we treat spectral projections and functional calculus of bounded decomposable operators,
with respect to
a fixed direct integral decomposition of Hilbert space
\[
\HEu=\int_\measspace^\oplus\HEu(\zeta)\,d\meas(\zeta).
\]

We let $\sigma(\cdot)$ denote the spectrum of an operator.
The following lemma is Proposition 1.1 of~\cite{Le74} and follows from Theorem~4.3 of~\cite{Az74}, which were proved independently.

\begin{lem}\label{lem:spec}
Suppose
\[
X=\int_\measspace^\oplus X(\zeta)\,d\meas(\zeta)
\]
is a bounded, decomposable operator.
Then for almost every $\zeta$, we have  $\sigma(X(\zeta))\subseteq\sigma(X)$.
\end{lem}

By appeal to the standard $*$-algebra operations, we have:
\begin{lem}\label{lem:Xnormal}
Let $X$ be a bounded, decomposable operator.
Then $X$ is a normal operator if and only if  $X(\zeta)$ is normal for almost all $\zeta$.
\end{lem}

We consider now the continuous functional calculus, which is quite straightforward to prove, and must be well known.
\begin{lem}\label{lem:contsFC}
Let $X$ be a bounded, normal, decomposable operator.
Using Lemmas~\ref{lem:spec} and~\ref{lem:Xnormal},
by redefining $X(\zeta)$ for $\zeta$ in a null set, if necessary, we may suppose $X(\zeta)$ is normal and has spectrum contained in $\sigma(X)$ for all $\zeta$.
Suppose $f:\sigma(X)\to\Cpx$ is a continuous function.
Then in the continuous functional calculus, we have
\[
f(X)=\int_\measspace^\oplus f(X(\zeta))\,d\meas(\zeta).
\]
\end{lem}
\begin{proof}
Take a sequence $(g_k)_{k=1}^\infty$ of polynomials in $z$ and $\zbar$ such that $g_k(z,\zbar)$ converges
uniformly to $f(z)$ for all $z\in\sigma(X)$.
Letting $\eps_k=\max_{z\in\sigma(X)}|f(z)-g_k(z,\zbar)|$, we have $\lim_{k\to\infty}\eps_k=0$.
But $\|f(X)-g_k(X,X^*)\|=\eps_k$ and for each $\zeta$, since $\sigma(X(\zeta))\subseteq\sigma(X)$,
we have $\|f(X(\zeta)-g_k(X(\zeta),X(\zeta)^*)\|\le\eps_k$
and from this we get (see~\ref{subsec:ddBdd}),
\[
\left\|\int_\measspace^\oplus f(X(\zeta))\,d\meas(\zeta)-\int_\measspace^\oplus g_k(X(\zeta),X(\zeta)^*)\,d\meas(\zeta)\right\|\le\eps_k.
\]
Since the $*$-algebra operations thread through decompositions, we have
\[
g_k(X,X^*)=\int_\measspace^\oplus g_k(X(\zeta),X(\zeta)^*)\,d\meas(\zeta).
\]
Taking $k\to\infty$ finishes the proof.
\end{proof}

We next consider spectral projections.
For a normal operator $X$ and a Borel subset $B$ of $\Cpx$, we will denote by $E_X(B)$ the corresponding spectral projection.
The following result is a special case of Proposition~1.4 of~\cite{Le74}.
However, for convenience, we provide a direct proof of this easier result.

\begin{prop}\label{prop:EB}
Suppose $X=\int_\measspace^\oplus X(\zeta)\,d\meas(\zeta)$ is a bounded, normal, decomposable operator 
and, as above, assume without loss of generality $X(\zeta)$ is normal and has spectrum contained in $\sigma(X)$ for all $\zeta$.
Let $B$ be a Borel subset of $\Cpx$.
Then
\begin{equation}\label{eq:EXB}
E_X(B)=\int_\measspace^\oplus E_{X(\zeta)}(B)\,d\meas(\zeta).
\end{equation}
\end{prop}
\begin{proof}
First suppose that $B$ is a nonempty open, bounded rectangle in $\Cpx$.
Let $(f_n)_{n=1}^\infty$ be an increasing sequence of continuous functions on $\Cpx$, each taking values in $[0,1]$ and vanishing outside of $B$
and such that $f_n$ converges pointwise to $1_B$ (the characteristic function of $B$) as $n\to\infty$.
By Lemma~\ref{lem:contsFC}, we have
\[
f_n(X)=\int_\measspace^\oplus f_n(X(\zeta))\,d\meas(\zeta).
\]
Since $f_n$ is increasing to $1_B$, by the spectral theorem, $f_n(X)$ converges in strong operator topology to $E_X(B)$.
Similarly, for every $\zeta$, $f_n(X(\zeta))$ converges in strong operator topology to $E_{X(\zeta)}(B)$, for all $\zeta$.
Thus, by Proposition~4 of Section II.2.3 of~\cite{Dix81},
$f_n(X)$ converges strongly to $\int_\measspace^\oplus E_{X(\zeta)}(B)\,d\meas(\zeta)$.
This yields the equality~\eqref{eq:EXB} when $B$ is an open rectangle.

We now show that the set $\beta$ of Borel sets $B$ with the property~\eqref{eq:EXB}
is a $\sigma$-algebra.
First, if $B \in \beta$, then 
\begin{multline*}
E_X(B^c)= 1-E_X(B) = 1-\int_\measspace^\oplus E_{X(\zeta)}(B)\,d\meas(\zeta) \\
= \int_\measspace^\oplus \left( 1-E_{X(\zeta)}(B)\right) d\meas(\zeta)=\int_\measspace^\oplus E_{X(\zeta)}(B^c) d\meas(\zeta),
\end{multline*}
so $B^c \in \beta.$  Now let $(B_n)_{n=1}^\infty$ be a sequence of sets from $\beta$.  For any $i,j \in \Nats$ we have 
\begin{align*}
E_X(B_i \cup B_j) &= E_X(B_i) + E_X(B_j) - E_X(B_i)E_X(B_j)\\ 
&= \int_\measspace^\oplus \left( E_{X(\zeta)}(B_i)+E_{X(\zeta)}(B_j)-E_{X(\zeta)}(B_i)E_{X(\zeta)}(B_j)\right)\,d\meas(\zeta) \\
&= \int_\measspace^\oplus E_{X(\zeta)}(B_i \cup B_j)\,d\meas(\zeta),
\end{align*}
so $B_i \cup B_j \in \beta$.  Hence $\beta$ is closed under finite unions.
Thus, for every $n$, we have
\[
E_X\big( \bigcup_{i=1}^n B_i \big)
= \int_\measspace^\oplus E_{X(\zeta)} \big( \bigcup_{i=1}^n B_i \big)\,d\meas(\zeta).
\]
But $E_X(\bigcup_{i=1}^n B_i)$ converges in strong operator topology to $E_X(\bigcup_{i=1}^\infty B_i)$,
and for each $\zeta$,
$E_{X(\zeta)}(\bigcup_{i=1}^n B_i)$ converges in strong operator topology to $E_{X(\zeta)}(\bigcup_{i=1}^\infty B_i)$.
So applying again Proposition~4 of Section II.2.3 of~\cite{Dix81}, we get
\[
E_X\big( \bigcup_{i=1}^\infty B_i \big)
= \int_\measspace^\oplus E_{X(\zeta)} \big( \bigcup_{i=1}^\infty B_i \big)\,d\meas(\zeta).
\]
Thus $\beta$ is a $\sigma$-algebra.

Since $\beta$ contains all of the bounded open rectangles, it is the whole Borel $\sigma$-algebra of $\Cpx$.
\end{proof}

From the above result, it is easy to show that an analogue of Lemma~\ref{lem:contsFC} holds for the Borel functional calculus.
\begin{prop}\label{prop:Linfty}
Let $X=\int_\measspace^\oplus X(\zeta)\,d\meas(\zeta)$ be a bounded, normal, decomposable operator.
Using Lemmas~\ref{lem:spec} and~\ref{lem:Xnormal},
by redefining $X(\zeta)$ for $\zeta$ in a null set, if necessary, we may suppose $X(\zeta)$ is normal and has spectrum contained in $\sigma(X)$ for all $\zeta$.
Suppose $f:\sigma(X)\to\Cpx$ is a bounded Borel function.
Then taking the Borel functional calculus, we have
\[
f(X)=\int_\measspace^\oplus f(X(\zeta))\,d\meas(\zeta).
\]
\end{prop}
\begin{proof}
Let $\eps>0$ and let $g=\sum_{j=1}^na_j1_{B_j}$ be a Borel measurable simple function such that $\sup_{z\in\sigma(X)}|f(z)-g(z)|<\eps$.
By Proposition~\ref{prop:EB}, we have
\[
g(X)=\int_\measspace^\oplus g(X(\zeta))\,d\meas(\zeta).
\]
But $\|g(X)-f(X)\|<\eps$.
Moreover, for all $\zeta$ we have $\|g(X(\zeta)-f(X(\zeta)\|<\eps$, so we get
\[
\left\|\int_\measspace^\oplus g(X(\zeta))\,d\meas(\zeta)-\int_\measspace^\oplus f(X(\zeta))\,d\meas(\zeta)\right\|\le\eps.
\]
This yields
\[
\left\|f(X)-\int_\measspace^\oplus f(X(\zeta))\,d\meas(\zeta)\right\|<2\eps.
\]
Letting $\eps\to0$ finishes the proof.
\end{proof}

\section{Functional calculus and affiliation for unbounded operators}
\label{sec:unbdd}

In this section, we prove a result about functional calculus for decomposable self-adjoint, possibly unbounded operators,
as well as a result about the polar decomposition of decomposable unbounded operators and one about affiliation to decomposable
von Neumann algebras.

\begin{lem}\label{lem:Cayley}
Let $T=\int_\measspace^\oplus T(\zeta)\,d\meas(\zeta)$ be a closed, (possibly unbounded), self-adjoint, decomposable operator.
Then the Cayley transform $(T+i)(T-i)^{-1}$ of $T$ is equal to the direct integral
\begin{equation}\label{eq:Cayley}
\int_\measspace^\oplus\big(T(\zeta) +i\big)\big(T(\zeta)-i\big)^{-1}\,d\meas(\zeta)
\end{equation}
of Cayley transforms.
\end{lem}
\begin{proof}
Note that the operator~\eqref{eq:Cayley} is unitary.
By evaluating at measurable vector fields $\zeta\mapsto x(t)$ belonging to $\dom(T)$, we have
\[
(T-i)x=\int_\measspace^\oplus\big(T(\zeta)-i\big)x(\zeta)\,d\meas(\zeta)
\]
and
\[
\left(\int_\measspace^\oplus\big(T(\zeta) +i\big)\big(T(\zeta)-i\big)^{-1}\,d\meas(\zeta)\right)(T-i)x
=\int_\measspace^\oplus\big(T(\zeta)+i\big)x(\zeta)\,d\meas(\zeta)=(T+i)x.
\]
Thus, the two unitary operators $(T+i)(T-i)^{-1}$ and
\[
\int_\measspace^\oplus\big(T(\zeta) +i\big)\big(T(\zeta)-i\big)^{-1}\,d\meas(\zeta),
\]
agree on a dense subset of $\HEu$, so they must be equal, as required.
\end{proof}

Now using the Cayley transform to go from unbounded self-adjoint operators to unitary operators,
we easily get the following analogues of Propositions~\ref{prop:EB} and~\ref{prop:Linfty}.
Here, for a Borel set $B$, we denote the corresponding spectral projection of also an unbounded self-adjoint operator $T$
by $E_T(B)$.

\begin{prop}\label{prop:Tsa}
Let $T=\int_\measspace^\oplus T(\zeta)\,d\meas(\zeta)$ be a closed, (possibly unbounded), self-adjoint, decomposable operator.
For every Borel subset $B\subset\Reals$, we have
\begin{equation}\label{eq:ETB}
E_T(B)=\int_\measspace^\oplus E_{T(\zeta)}(B)\,d\meas(\zeta).
\end{equation}
Moreover, for every (possibly unbounded) Borel measurable function $f:\Reals\to\Reals$, we have
\begin{equation}\label{eq:fT}
f(T)=\int_\measspace^\oplus f(T(\zeta))\,d\meas(\zeta).
\end{equation}
\end{prop}
\begin{proof}
Consider the map $h:\Reals\to\Tcirc$ given by $h(t)=\frac{t+i}{t-i}$.
Let $U=(T+i)(T-i)^{-1}$ be the Cayley transform of $T$ and let $U(\zeta)=(T(\zeta)+i)(T(\zeta)-i)^{-1}$.
Then for all $\zeta$ we have
\[
E_T(B)=E_U(h(B))\qquad\text{and}\qquad E_{T(\zeta)}(B)=E_{U(\zeta)}(h(B)).
\]
Thus, applying Proposition~\ref{prop:EB} to $U$ and $h(B)$ yields~\eqref{eq:ETB}.
Now, by approximating $f$ in norm with simple Borel measurable functions, as was done for bounded operators in the proof of
Proposition~\ref{prop:Linfty},  we obtain~\eqref{eq:fT}.
\end{proof}

Nussbaum proved (Theorem~5 of~\cite{Nu64}) that given a densely defined, decomposable, (possibly unbounded) closed operator
\[
T=\int_\measspace^\oplus T(\zeta)\,d\meas(\zeta),
\]
its absolute value is the direct integral of absolute values:
\begin{equation}\label{eq:|T|}
|T|=\int_\measspace^\oplus |T(\zeta)|\,d\meas(\zeta).
\end{equation}

\begin{prop}\label{prop:polar}
With $T$ as above, let $T=V|T|$ be the polar decomposition of $T$.
Then the polar part $V$ is decomposable and we have
\begin{equation}\label{eq:V}
V=\int_\measspace^\oplus V(\zeta)\,d\meas(\zeta),
\end{equation}
where $V(\zeta)$ is the polar part in the polar decomposition $T(\zeta)=V(\zeta)|T(\zeta)|$ of $T(\zeta)$.
\end{prop}
\begin{proof}
Let $W$ be the bounded, decomposable operator 
defined by the right-hand-side of~\eqref{eq:V}.
Then $W$ is a partial isometry.
By evaluating on vector fields $x\in\HEu$ in $\dom(T)=\dom(|T|)$, and using~\eqref{eq:|T|},
we find
\[
|T|x=\int_\measspace^\oplus |T(\zeta)|x(\zeta)\,d\meas(\zeta)
\]
and
\[
W|T|x=\int_\measspace^\oplus V(\zeta)|T(\zeta)|x(\zeta)\,d\meas(\zeta)=\int_\measspace^\oplus T(\zeta)x(\zeta)\,d\meas(\zeta)=Tx.
\]
Thus we have
\begin{equation}\label{eq:Tpolar}
W|T|=T.
\end{equation}
Moreover, $V(\zeta)^*V(\zeta)$ is the range projection $E_{|T(\zeta)|}((0,\infty))$ of $|T(\zeta)|$.
Thus,
\[
W^*W=\int_\measspace^\oplus V(\zeta)^*V(\zeta)\,d\meas(\zeta)=\int_\measspace^\oplus E_{|T(\zeta)|}((0,\infty))\,d\meas(\zeta)=E_{|T|}((0,\infty)),
\]
where the last equality is provided by Proposition~\ref{prop:Tsa}.
This, together with~\eqref{eq:Tpolar}, implies that $T=W|T|$ is the polar decomposition of $T$.
\end{proof}

Recall that for a closed, densely defined operator $T$ in $\HEu$ and a von Neumann algebra $\Mcal\subseteq B(\HEu)$,
we say that $T$ is {\em affiliated} to $\Mcal$ if, letting $T=V|T|$ denote the polar decomposition of $T$,
we have $V\in\Mcal$ and $E_{|T|}(B)\in\Mcal$
for every Borel subset $B$ of $\Reals$.

The following is the analogue for unbounded operators
of the fundamental fact about decompositions of von Neumann algebras stated in~\ref{subsec:vN}.
\begin{prop}\label{prop:decompaffil}
Suppose
\[
\Mcal=\int_\measspace^\oplus \Mcal(\zeta)\,d\meas(\zeta)
\]
is  decomposable von Neumann algebra (see~\ref{subsec:vN}).
Let $T$ be a closed (possibly unbounded) operator in $\HEu$.
Then $T$ is affiliated to $\Mcal$ if and only if (a) $T$ is decomposable and
(b) writing out the decompsition as
\begin{equation}\label{eq:Tdecomp}
T=\int_\measspace^\oplus T(\zeta)\,d\meas(\zeta),
\end{equation}
we have that $T(\zeta)$ is affiliated to $\Mcal(\zeta)$ for almost every $\zeta$.
\end{prop}
\begin{proof}
First we show $\Leftarrow$.
Suppose $T$ is decomposable and is written as in~\eqref{eq:Tdecomp}.
Let $T=V|T|$ and $T(\zeta)=V(\zeta)|T(\zeta)|$ be the polar decompositions.
For almost every $\zeta$ we have $V(\zeta)\in\Mcal(\zeta)$;
using Proposition~\ref{prop:polar}, we have $V\in\Mcal$.
Similarly, for every Borel subset $B\subseteq\Reals$, we have $E_{|T(\zeta)|}(B)\in\Mcal(\zeta)$ for almost every $\zeta$, so using Proposition~\ref{prop:Tsa},
we find $E_{|T|}(B)\in\Mcal$.
Thus, $T$ is affiliated to $\Mcal$.

To show $\Rightarrow$, we suppose $T$ is affiliated to $\Mcal$.
Let $T=V|T|$ be the polar decomposition of $T$.
Since $V\in\Mcal$ and all spectral projections $E_{|T|}(B)$ are in $\Mcal$,
they all commute with all the diagonalizable operators;
from this, we easily see that $T$ permutes with all diagonalizable operators.
By Nussbaum's Corollary 4 of~\cite{Nu64}, $T$ is decomposable;
we write it as in~\eqref{eq:Tdecomp}.
Let $T(\zeta)=V(\zeta)|T(\zeta)|$ be the polar decomposition.
Since $V\in\Mcal$, using Proposition~\ref{prop:polar} we get $V(\zeta)\in\Mcal(\zeta)$ for almost every $\zeta$.
Similarly, but using Proposition~\ref{prop:Tsa}, for every Borel set $B$,
since $E_{|T|}(B)\in\Mcal$, there is a null set $N_B$ such that for all $\zeta\notin N_B$, we have $E_{|T(\zeta)|}(B)\in\Mcal(\zeta)$.
Let $N$ be the union of the sets $N_B$ as $B$ ranges over the open intervals with rational endpoints in $\Reals$.
Then $N$ is a null set and for all $\zeta\notin N$ we have $E_{|T(\zeta)|}((a,b))\in\Mcal(\zeta)$ for all rational numbers $a<b$.
From this, we deduce $E_{|T(\zeta)|}(B)\in\Mcal(\zeta)$ for all Borel subsets $B\subseteq\Reals$.
Thus, we have that $T(\zeta)$ is affiliated to $\Mcal(\zeta)$ for almost every $\zeta$.
\end{proof}

\section{Tracial von Neumann algebras and Brown measure}
\label{sec:Brown}

In this section, we will specialize to the case of operators in or affiliated to tracial von Neumann algebras,
by which we mean, pairs $(\Mcal,\tau)$ consisting of a von Neumann algebra $\Mcal$ and a fixed normal, faithful, tracial state $\tau$ on it.
Recall that, given such a pair, we let $\exp(\Lc^1)(\Mcal,\tau)$ denote the bimodule of closed operators $T$ affiliated to $\Mcal$ such
that $\tau(\log^+(|T|))<\infty$.

Here is a technical lemma that we will need later; it is convenient to prove it here.
\begin{lem}\label{lem:cont}
Let $T\in\exp(\mathcal{L}_1)(\mathcal{M},\tau)$.
Then the mapping $\lambda\mapsto\Delta_\tau(|T-\lambda|^2+1)$ ($\lambda\in\Cpx$) is continuous.
\end{lem}
\begin{proof}
By shifting $T,$ it suffices to prove that our mapping is continuous at $0$.
To see this, note that
\begin{align*}
\Delta(|T-\lambda|^2+1)&=\Delta(|T|^2+1)\Delta((1+|T|^2)^{-\frac12}(|T-\lambda|^2+1)(1+|T|^2)^{-\frac12}) \\
&=\Delta(|T|^2+1)\Delta(1+(1+|T|^2)^{-\frac12}(|T-\lambda|^2-|T|^2)(1+|T|^2)^{-\frac12}).
\end{align*}
It will, thus, suffice to show
\begin{equation}\label{eq:lam0}
\lim_{\lambda\to0}\Delta(1+(1+|T|^2)^{-\frac12}(|T-\lambda|^2-|T|^2)(1+|T|^2)^{-\frac12})=1.
\end{equation}
It is immediate that
\[
|T-\lambda|^2-|T|^2=|\lambda|^2-\lambda T^*-\bar{\lambda}T=|\lambda|^2-\lambda|T|U^*-\bar{\lambda}U|T|,
\]
where $T=U|T|$ is the polar decomposition.
Thus,
\begin{multline*}
(1+|T|^2)^{-\frac12}(|T-\lambda|^2-|T|^2)(1+|T|^2)^{-\frac12}= \\
\begin{aligned}[t]=|\lambda|^2&(1+|T|^2)^{-1}-\lambda\left(\frac{|T|}{(1+|T|^2)^{\frac12}}\right)U^*\left(\frac1{(1+|T|^2)^{\frac12}}\right) \\[1ex]
&-\bar{\lambda}\left(\frac1{(1+|T|^2)^{\frac12}}\right)U\left(\frac{|T|}{(1+|T|^2)^{\frac12}}\right).
\end{aligned}
\end{multline*}
Thus, we have the estimate of operator norm
\[
\Big\|(1+|T|^2)^{-\frac12}(|T-\lambda|^2-|T|^2)(1+|T|^2)^{-\frac12}\Big\|\leq 2|\lambda|+|\lambda|^2.
\]
So when $|\lambda|\le\frac13$, we have
\begin{multline*}
\log(1-2|\lambda|-|\lambda|^2)\le\log\Delta(1+(1+|T|^2)^{-\frac12}(|T-\lambda|^2-|T|^2)(1+|T|^2)^{-\frac12}) \\
\le\log(1+2|\lambda|+|\lambda|^2),
\end{multline*}
which proves~\eqref{eq:lam0}.
This concludes the proof.
\end{proof}

For the remainder of this section, we suppose
$\Mcal\subseteq B(\HEu)$ is a von Neumann algebra equipped with a normal, faithful tracial state $\tau$
and that $\Mcal\subseteq\Ec$ consists of decomposable operators.
Using Dixmier's reduction theory~\cite{Dix81} (described in~\ref{subsec:finvN}),
and by modifying the measure $\meas$ to be a probability measure, we may write
\[
\Mcal=\int_\measspace^\oplus\Mcal(\zeta)\,d\meas(\zeta),\quad\text{and}\quad
\tau=\int_\measspace^\oplus\tau_\zeta\,d\meas(\zeta),
\]
for tracial von Neumann algebras $(\Mcal(\zeta),\tau_\zeta)$, with $\Mcal(\zeta)\subseteq B(\HEu(\zeta))$.
By Proposition~\ref{prop:decompaffil}
if $T$ is affiliated to $\Mcal$, then $T$ is decomposable and may be written 
\begin{equation}\label{eq:Tdecomp2}
T=\int_\measspace^\oplus T(\zeta)\,d\meas(\zeta),
\end{equation}
with $T(\zeta)$ affiliated to $\Mcal(\zeta)$ for almost every $\zeta$.

For an element $T\in\exp(\Lc^1)(\Mcal,\tau)$, we let $\nu_T$ denote the Brown measure of $T$.
For any self-adjoint, closed operator $T$ affiliated to $\Mcal$, we let $\mu_T$ denote the distribution of $T$,
namely, $\tau$ composed with spectral measure of $T$.
In fact, when $T\in\exp(\Lc^1)(\Mcal,\tau)$ is self-adjoint, we have $\nu_T=\mu_T$ (this follows immediately from the characterization provided by
Equations~\eqref{eq:char1} and~\eqref{eq:char2}) so there would be no conflict in using the same notation for both;
but for clarity of meaning, we will distinguish them.

In this setting, Proposition~\ref{prop:Tsa} yields the following formula for spectral distributions of self-adjoint (possibly unbounded) operators.
\begin{prop}\label{prop:nuT}
Let $T$ be self-adjoint and affiliated to $\Mcal$.
Then for every Borel subset $B$ of $\Reals$, the function $\zeta\mapsto\mu_{T(\zeta)}(B)$ is measurable and
\[
\mu_T(B)=\int_\measspace\mu_{T(\zeta)}(B)\,d\meas(\zeta).
\]
\end{prop}

We let $\Lc^1(\Mcal,\tau)$ denote the set of all closed operators affiliated to $\Mcal$ such that $\tau(|T|)<\infty$.
\begin{lem}\label{lem:L1}
Suppose $T\in\Lc^1(\Mcal,\tau)$ and $T\ge0$;
use the decomposition~\eqref{eq:Tdecomp2}.
Then $T(\zeta)\in\Lc^1(\Mcal(\zeta),\tau_\zeta)$ for almost every $\zeta$
and
\begin{equation}\label{eq:tauT}
\tau(T)=\int_\measspace\tau_\zeta(T(\zeta))\,d\meas(\zeta).
\end{equation}
\end{lem}
\begin{proof}
We have $T(\zeta)\ge0$ for almost every $\zeta$.
Since the decompositions of $T$ and $\tau$ are measurable, the function $\zeta\mapsto\tau_\zeta(T(\zeta))$ is measurable.
Let $(f_n)_{n=1}^\infty$ be an increasing sequence of simple functions, each having finitely many values,
that converges pointwise to the identity function $t\mapsto t$ on $[0,\infty)$.
Then $\tau(f_n(T))$ is increasing in $n$ and converges to $\tau(T)$ while for every $\zeta$ such that $T(\zeta)\ge0$,
the sequence $\tau_\zeta(f_n(T(\zeta)))$ is increasing in $n$ and convergest to $\tau_\zeta(T(\zeta))$.
Now fixing $n$ and writing $f_n=\sum_{k=1}^ma_k1_{B_k}$ for some $a_k\ge0$ and some Borel sets $B_k$,
using Proposition~\ref{prop:nuT}, we find
\[
\tau(f_n(T))=\sum_ka_k\mu_T(B_k)=\sum_ka_k\int_\measspace\mu_{T(\zeta)}(B_k)\,d\meas(\zeta)=\int_\measspace\tau_\zeta(f_n(T(\zeta)))\,d\meas(\zeta).
\]
Letting $n\to\infty$, the Monotone Convergence Theorem implies the equality~\eqref{eq:tauT}.
This, in turn, impies $\tau_\zeta(T(\zeta))<\infty$ for almost every $\zeta$.
\end{proof}

Now we turn to the $\exp(\Lc^1)$ class and the Fuglede--Kadison determinant.

\begin{lem}\label{lem:intFK}
Let $T\in\exp(\Lc^1)(\Mcal,\tau)$ and use the decomposition~\eqref{eq:Tdecomp2}.
Then $T(\zeta)\in\exp(\Lc^1)(\Mcal(\zeta),\tau_\zeta)$ for almost every $\zeta$.
Moreover, we have
\begin{align}
\tau(\log^+(|T|))&=\int_\measspace^\oplus\tau_\zeta\big(\log^+(|T(\zeta)|)\big)\,d\meas(\zeta) \label{eq:intlog+bd} \\
\log \Delta_\tau (T)&=\int_\measspace\log\Delta_{\tau_\zeta}(T(\zeta))\,d\meas(\zeta). \label{eq:intFK}
\end{align}
\end{lem}
\begin{proof}
By Theorem~5 of~\cite{Nu64} --- see Equation~\eqref{eq:|T|} ---
we may without loss of generality assume $T\ge0$, which entails $T(\zeta)\ge0$ for almost every $\zeta$.
Now using Proposition~\ref{prop:Tsa}, we get
\[
\log^+(T)=\int_\measspace^\oplus\log^+(T(\zeta))\,d\meas(\zeta).
\]
Since $T\in\exp(\Lc^1)(\Mcal,\tau)$, we have $\log^+(T)\in\Lc^1(\Mcal,\tau)$.
Now Lemma~\ref{lem:L1} yields~\eqref{eq:intlog+bd}
and we deduce $\log^+(T(\zeta))\in\Lc^1(\Mcal(\zeta),\tau_\zeta)$, namely, $T(\zeta)\in\exp(\Lc^1)(\Mcal(\zeta),\tau_\zeta)$,
for almost every $\zeta$.

Now we show~\eqref{eq:intFK}.
Let $\eps>0$.
Using the function $f_\eps(t)=\log(t+\eps)$ ($t\ge0$) and using Proposition~\ref{prop:Tsa} to apply the functional calculus to $T$, we get
\begin{equation}\label{eq:logTeps}
\log(T+\eps)=\int_\measspace^\oplus\log(T(\zeta)+\eps)\,d\meas(\zeta).
\end{equation}
Now Lemma~\ref{lem:L1} applies (if we first add $-\log\eps$ to both sides of~\eqref{eq:logTeps} to make the operators positive)
and we have
\[
\tau\big(\log(T+\eps)\big)=\int_\measspace\tau_\zeta\big(\log(T(\zeta)+\eps)\big)\,d\meas(\zeta).
\]
Letting $\eps\to0$ and using the Monotone Convergence Theorem, we get
\[
\log\Delta_\tau(T)=\tau\big(\log(T)\big)=\int_\measspace\tau_\zeta\big(\log(T(\zeta))\big)\,d\meas(\zeta)=\int_\measspace\log\Delta_{\tau_\zeta}(T(\zeta))\,d\meas(\zeta),
\]
as required.
\end{proof}

Recall that, for $T\in\exp(\Lc^1)(\Mcal,\tau)$, we let $\nu_T$ denote the Brown measure of $T$.

\begin{lem}\label{lem:numeas}
Let $T\in\exp(\Lc^1)(\Mcal,\tau)$ and use the decomposition~\eqref{eq:Tdecomp2}.
Then for every Borel subset $B\subseteq\Cpx$ the mapping $\zeta\mapsto\nu_{T(\zeta)}(B)$ is measurable.
\end{lem}
\begin{proof}
By Lemma~\ref{lem:intFK}, $T(\zeta)\in\exp(\Lc^1)(\Mcal(\zeta),\tau_\zeta)$ for almost all $\zeta$, and we will confine ourselves to such $\zeta$.
It will suffice to prove measurability when $B$ is an open, bounded rectangle in $\Cpx$, for the collection of such sets generates the Borel $\sigma$-algebra.
Fix a sequence $\{f_n\}_{n\geq0}$ of Schwartz functions having support in $B$ and increasing pointwise to the characteristic function of $B$.
Then by the Monotone Convergence Theorem, we have
\[
\nu_{T(\zeta)}(B)=\lim_{n\to\infty}\int_{\Cpx}f_n(\lambda)d\nu_{T(\zeta)}(\lambda).
\]
By definition of the Brown measure, we have
\[
\int_{\Cpx}f_n(\lambda)d\nu_{T(\zeta)}(\lambda)=\frac1{2\pi}\int_{\Cpx}\tau_{\zeta}(\log(|T(\zeta)-\lambda|))\nabla^2f_n(\lambda)\,d\lambda,
\]
where $d\lambda$ means Lebesgue measure on $\Cpx$.
Note that $\tau_\zeta(\log|T(\zeta)-\lambda|)$ is bounded above for $\lambda$ in compact subsets of $\Cpx$.
Fixing $n$ for the moment and writing
$\nabla^2f_n(\lambda)=h_{1}-h_{2},$ where $h_{1}$ and $h_{2}$ are positive Schwartz functions,
it follows that both of the integrals
\[
\int_{\Cpx}\tau_{\zeta}(\log(|T(\zeta)-\lambda|))h_1(\lambda)\,d\lambda\quad\text{and}\quad
\int_{\Cpx}\tau_{\zeta}(\log(|T(\zeta)-\lambda|))h_2(\lambda)\,d\lambda
\]
are finite.
It follows from the Monotone Convergence Theorem that
\begin{align*}
\int_{\Cpx}\tau_{\zeta}(\log(|T(\zeta)-\lambda|))h_{1}(\lambda)\,d\lambda
&=\frac12\lim_{m\to\infty}\int_{\Cpx}\tau_{\zeta}(\log(|T(\zeta)-\lambda|^2+\frac1{m^2}))h_{1}(\lambda)\,d\lambda, \\[1ex]
\int_{\Cpx}\tau_{\zeta}(\log(|T(\zeta)-\lambda|))h_{2}(\lambda)\,d\lambda
&=\frac12\lim_{m\to\infty}\int_{\Cpx}\tau_{\zeta}(\log(|T(\zeta)-\lambda|^2+\frac1{m^2}))h_{2}(\lambda)\,d\lambda.
\end{align*}
Thus, we have
\[
\int_{\Cpx}\tau_{\zeta}(\log(|T(\zeta)-\lambda|))\nabla^2f_n(\lambda)\,d\lambda
=\frac12\lim_{m\to\infty}\int_{\Cpx}\tau_{\zeta}(\log(|T(\zeta)-\lambda|^2+\frac1{m^2}))\nabla^2f_n(\lambda)\,d\lambda
\]
and, since each $\nabla^2f_n$ vanishes outside of the rectangle $B$,
\[
\nu_{T(\zeta)}(B)=\frac1{4\pi}\lim_{n\to\infty}\lim_{m\to\infty}\int_B\tau_{\zeta}(\log(|T(\zeta)-\lambda|^2+\frac1{m^2}))\nabla^2f_n(\lambda)d\lambda.
\]

By Lemma~\ref{lem:cont}, the mapping
\[
\lambda\mapsto \tau_{\zeta}(\log(|T(\zeta)-\lambda|^2+\frac1{m^2}))\nabla^2f_n(\lambda)
\]
is continuous and, therefore, is Riemann integrable over $B$.
Thus,
\begin{multline*}
\int_B\tau_{\zeta}(\log(|T(\zeta)-\lambda|^2+\frac1{m^2}))\nabla^2f_n(\lambda)d\lambda \\
=\lim_{k\to\infty}\frac1{k^2}\sum_{\lambda\in\frac1k(\Ints+i\Ints)}\tau_{\zeta}(\log(|T(\zeta)-\lambda|^2+\frac1{m^2}))\nabla^2f_n(\lambda),
\end{multline*}
where the sum is actually finite.
Thus,
\[
\nu_{T(\zeta)}(B)=\frac1{4\pi}\lim_{n\to\infty}\lim_{m\to\infty}\lim_{k\to\infty}\frac1{k^2}
\sum_{\lambda\in\frac1k(\Ints+i\Ints)}\tau_{\zeta}(\log(|T(\zeta)-\lambda|^2+\frac1{m^2}))\nabla^2f_n(\lambda).
\]
Because the decompositions of $T$ and of $\tau$ are measurable, for each fixed $\lambda$ the mapping
\[
\zeta\to\tau_{\zeta}(\log(|T(\zeta)-\lambda|^2+\frac1{m^2}))
\]
is measurable.
Since the pointwise limit of the sequence of measurable functions is again a measurable function, the lemma is proved.
\end{proof}

Here is the main theorem about decomposition of Brown measure.
\begin{thm}\label{thm:decompBrown}
Let $T\in\exp(\Lc^1)(\Mcal,\tau)$ and write 
\[
T=\int_\measspace^\oplus T(\zeta)\,d\meas(\zeta).
\]
Then the Brown measure $\nu_T$ of $T$ is given by
\begin{equation}\label{eq:rhofornuT}
\nu_T(B)=\int_\measspace\nu_{T(\zeta)}(B)\,d\meas(\zeta)
\end{equation}
for every Borel subset $B\subseteq\Cpx$.
\end{thm}
\begin{proof}
By Lemma~\ref{lem:numeas}, the right-hand-side of~\eqref{eq:rhofornuT} defines a probability measure on $\Cpx$, which we will denote by the symbol $\rho$.
We will show that $\rho$ satisfies
\begin{gather}
\int_{\Cpx}\log^+|z|\,d\rho(z)<\infty \label{eq:rhochar1} \\
\int_{\Cpx}\log|z-\lambda|\,d\rho(z)=\log\Delta_\tau(T-\lambda)\qquad(\lambda\in\Cpx). \label{eq:rhochar2}
\end{gather}
From the uniqueness property
of Brown measure expressed with Equations~\eqref{eq:char1} and \eqref{eq:char2}, this will imply $\rho=\nu_T$.

To prove~\eqref{eq:rhochar1},
let $f_n$ be an increasing sequence of simple functions on $\Cpx$, each taking only finitely many values, that converges pointwise
to the function $w\mapsto\log^+(w)$.
For each $n$, we have
\begin{equation}\label{eq:fn}
\int_\Cpx f_n(w)\,d\rho(w)=\int_\measspace\int_\Cpx f_n(w)\,d\nu_{T(\zeta)}(w)\,d\meas(\zeta).
\end{equation}
Applying the Monotone Convergence Theorem, we get
\[
\int_\Cpx\log^+|w|\,d\rho(w)=\int_\measspace\int_\Cpx\log^+(w)\,d\nu_{T(\zeta)}(w)\,d\meas(\zeta).
\]
By Lemma~2.20 of~\cite{HS07}, for each $\zeta$, we have
\begin{equation*}
\int_\Cpx\log^+(|w|)\,d\nu_{T(\zeta)}(w)\le\tau_\zeta\big(\log^+(|T(\zeta)|)\big).
\end{equation*}
Since $T\in\exp(\Lc^1)(\Mcal,\tau)$, using Lemma~\ref{lem:intFK}, we have
\begin{equation*}
\int_\measspace\tau_\zeta\big(\log^+(|T(\zeta)|)\big)\,d\meas(\zeta)<\infty.
\end{equation*}
This implies~\eqref{eq:rhochar1}.

Now fix $\lambda\in\Cpx$ and $\eps>0$ and let $(f_n)_{n=1}^\infty$ be an increasing sequence of simple Borel measurable functions on $\Cpx$,
each taking only finitely many values,
that converges pointwise to the function $w\mapsto\log(|w-\lambda|+\eps)$.
Again we have~\eqref{eq:fn}.
Using the Monotone Convergence Theorem and taking $n\to\infty$ we get
\begin{equation*}
\int_\Cpx\log(|w-\lambda|+\eps)\,d\rho(w)=\int_\measspace\int_\Cpx\log(|w-\lambda|+\eps)\,d\nu_{T(\zeta)}(w)\,d\meas(\zeta).
\end{equation*}
Using~\eqref{eq:rhochar1}, we see that the left-hand-side above is not $+\infty$.
Thus, letting $\eps\to0$ and using the Monotone Convergence Theorem, we get
\begin{multline*}
\int_\Cpx\log(|w-\lambda|)\,d\rho(w)=\int_\measspace\int_\Cpx\log(|w-\lambda|)\,d\nu_{T(\zeta)}(w)\,d\meas(\zeta) \\
=\int_\measspace\log\Delta_{\tau_\zeta}(T(\zeta)-\lambda)\,d\meas(\zeta).
\end{multline*}
From~\eqref{eq:intFK} of Lemma~\ref{lem:intFK}, we get~\eqref{eq:rhochar2}.
\end{proof}


\begin{thebibliography}{9}

\bib{Az74}{article}{
   author={Azoff, Edward A.},
   title={Spectrum and direct integral},
   journal={Trans. Amer. Math. Soc.},
   volume={197},
   date={1974},
   pages={211--223},
}

\bib{B86}{article}{
  author={Brown, Lawrence G.},
  title={Lidskii's theorem in the type II case},
  conference={
    title={Geometric methods in operator algebras},
    address={Kyoto},
    date={1983},
  },
  book={
    series={Pitman Res. Notes Math. Ser.}, 
    volume={123},
    publisher={Longman Sci. Tech.},
    address={Harlow},
    date={1986},
  },
  pages={1--35},
}

\bib{Dix81}{book}{
   author={Dixmier, Jacques},
   title={Von Neumann algebras},
   series={North-Holland Mathematical Library},
   volume={27},
   publisher={North-Holland Publishing Co., Amsterdam-New York},
   date={1981},
}

\bib{DSZ}{article}{
  author={Dykema, Ken},
  author={Sukochev, Fedor},
  author={Zanin, Dmitriy},
  title={A decomposition theorem in II$_1$--factors},
  journal={J. reine angew. Math.},
  status={to appear},
  eprint={http://arxiv.org/abs/1302.1114}
}

\bib{DSZ:unbdd}{article}{
  author={Dykema, Ken},
  author={Sukochev, Fedor},
  author={Zanin, Dmitriy},
  title={An upper triangular decomposition theorem for some unbounded operators affiliated to II$_1$-factors},
  status={preprint},
}

\bib{FK52}{article}{
  author={Fuglede, Bent},
  author={Kadison, Richard V.},
  title={Determinant theory in finite factors},
  journal={Ann. of Math.},
  volume={55},
  year={1952},
  pages={520--530},
}

\bib{GGST12}{article}{
   author={Gesztesy, Fritz},
   author={Gomilko, Alexander},
   author={Sukochev, Fedor},
   author={Tomilov, Yuri},
   title={On a question of A. E. Nussbaum on measurability of families of
   closed linear operators in a Hilbert space},
   journal={Israel J. Math.},
   volume={188},
   date={2012},
   pages={195--219},
}

\bib{HS07}{article}{
  author={Haagerup, Uffe},
  author={Schultz, Hanne},
  title={Brown measures of unbounded operators affiliated with a finite von Neumann algebra},
  journal={Math. Scand.},
  volume={100},
  date={2007},
  pages={209--263},
}

\bib{Le74}{article}{
   author={Lennon, M. J. J.},
   title={Direct integral decomposition of spectral operators},
   journal={Math. Ann.},
   volume={207},
   date={1974},
   pages={257--268},
}

\bib{Nu64}{article}{
   author={Nussbaum, A. E.},
   title={Reduction theory for unbounded closed operators in Hilbert space},
   journal={Duke Math. J.},
   volume={31},
   date={1964},
   pages={33--44},
}

\end{thebibliography}
\end{document}